\documentclass{amsart}
\usepackage{amssymb}
\usepackage{graphicx}
\usepackage{amscd}
\usepackage{amsmath}
\usepackage{amsfonts}

\providecommand{\U}[1]{\protect\rule{.1in}{.1in}}
\providecommand{\U}[1]{\protect\rule{.1in}{.1in}}
\newtheorem{theorem}{Theorem}
\theoremstyle{plain}

\newtheorem{lemma}{Lemma}

\numberwithin{equation}{section}

\input{tcilatex}

\begin{document}
\title{Multiplicative semiderivations on ideals in semiprime rings}
\author{\"{O}znur G\"{o}lba\c{s}\i}
\address{Cumhuriyet University, Faculty of Science, Department of
Mathematics, Sivas - TURKEY}
\email{ogolbasi@cumhuriyet.edu.tr}
\author{Onur A\u{g}\i rt\i c\i}
\email{\"{y} }

\begin{abstract}
Let $R$ be a semiprime ring and $I$ is a nonzero ideal of $R.$ A mapping $%
d:R\rightarrow R$ is called a multiplicative semiderivation if there exists
a function $g:R\rightarrow R$ such that (i) $%
d(xy)=d(x)g(y)+xd(y)=d(x)y+g(x)d(y)$ and (ii) $d(g(x))=g(d(x))$ hold for all 
$x,y\in R.$ In the present paper, we shall prove that $[x,d(x)]=0,$ for all $%
x\in I$ if any one of the following holds: i) $d([x,y])=0,$ ii) $d(xoy)=0,$
iii) $d(xy)\pm xy=0,$ iv) $d(xy)\pm yx=0,$ v) $d(x)d(y)\pm xy=0,$ vi) $%
d(x)d(y)\pm yx=0,$ vii) $d(xy)=d(x)d(y),$ viii) $d(xy)=d(y)d(x),$ for all $%
x,y\in I.$
\end{abstract}

\maketitle

\section{Introduction}

Let $R$ will be an associative ring with center $Z.$ For any $x,y\in R$ the
symbol $[x,y]$ represents the Lie commutator $xy-yx$ and the Jordan product $%
xoy=xy+yx.$ Recall that a ring $R$ is prime if for $x,y\in R,$ $xRy=0$
implies either $x=0$ or $y=0$ and $R$ is semiprime if for $x\in R,$ $xRx=0$
implies $x=0.$ It is clear that every prime ring is semiprime ring.

The study of derivations in prime rings was initiated by E. C. Posner in 
\cite{posner}. An additive mapping $d:R\rightarrow R$ is called a derivation
if $d(xy)=d(x)y+xd(y)$ holds for all $x,y\in R.$ In \cite{bergen}, J. Bergen
has introduced the notion of semiderivation of a ring $R$ which extends the
notion of derivations of a ring $R.$ An additive mapping $d:R\rightarrow R$
is called a semiderivation if there exists a function $g:R\rightarrow R$
such that (i) $d(xy)=d(x)g(y)+xd(y)=d(x)y+g(x)d(y)$ and (ii) $%
d(g(x))=g(d(x)) $ hold for all $x,y\in R.$\ In case $g$ is an identity map
of $R,$ then all semiderivations associated with $g$ are merely ordinary
derivations. On the other hand, if $g$ is a homomorphism of $R$ such that $%
g\neq 1,$ then $f=g-1$ is a semiderivation which is not a derivation. In
case $R$ is prime and $d\neq 0,$ it has been shown by Chang \cite{chang}
that $g$ must necessarily be a ring endomorphism.

Many authors have studied commutativity of prime and semiprime rings
admitting derivations, generalized derivations and semiderivations which
satisfy appropriate algebraic conditions on suitable subsets of the rings.
In \cite{daif-1}, the notion of multiplicative derivation was introduced by
Daif motivated by Martindale in \cite{martindale}. $d:R\rightarrow R$ is
called a multiplicative derivation if $d(xy)=d(x)y+xd(y)$ holds for all $%
x,y\in R.$ These maps are not additive. In \cite{goldman}, Goldman and Semrl
gave the complete description of these maps. We have $R=C[0,1],$ the ring of
all continuous (real or complex valued) functions and define a mapping $%
d:R\rightarrow R$ such as%
\begin{equation*}
d(f)(x)=\left\{ 
\begin{array}{c}
f(x)\log \left\vert f(x)\right\vert ,\text{ \ \ \ \ }f(x)\neq 0 \\ 
0,\text{ \ \ \ \ \ \ \ \ \ \ \ \ \ \ \ \ \ \ \ \ \ otherwise}%
\end{array}%
\right\} .
\end{equation*}%
It is clear that $d$ is multiplicative derivation, but $d$ is not additive.
Recently, some well-known results concerning semiprime rings have been
proved for multiplicative derivations.

Inspired by the definition multiplicative derivation, we can define the
notion of multiplicative semiderivation such as: A mapping $d:R\rightarrow R$
is called a multiplicative semiderivation if there exists a function $%
g:R\rightarrow R$ such that (i) $d(xy)=d(x)g(y)+xd(y)=d(x)y+g(x)d(y)$ and
(ii) $d(g(x))=g(d(x))$ hold for all $x,y\in R.$\ Hence, one may observe that
the concept of multiplicative semiderivations includes the concept of
derivations and the left multipliers (i.e., $d(xy)=d(x)y$ for all $x,y\in R$%
). So, it should be interesting to extend some results concerning these
notions to multiplicative semiderivations. Every derivation is a
multiplicative semiderivation. But the converse is not ture in general.

In \cite{daif-2}, Daif and Bell proved that $R$ is semiprime ring, $U$ is a
nonzero ideal of $R$ and $d$ is a derivation of $R$ such that $d([x,y])=\pm
\lbrack x,y],$ for all $x,y\in U,$ then $R$ contains a nonzero central
ideal. On the other hand, in \cite{rehman}, Ashraf and Rehman showed that $R$
is prime ring with a nonzero ideal $U$ of $R$ and $d$ is a derivation of $R$
such that $d(xy)\pm xy\in Z,$ for all $x,y\in U,$ then $R$ is commutative.
Also, Bell and Kappe proved that a derivation $d$ of a prime ring $R$ acts
as homomorphism or anti-homomorphism on a nonzero right ideal of $R,$ then $%
d=0$ on $R$ in \cite{bell}. Motivated by these works, we consider similar
situations for multiplicative semiderivation on nonzero ideal of semiprime
ring $R.$

\section{Results}

Throughout the paper, $R$ will be semiprime ring and $I$\ be a nonzero ideal
of $R$ and $d$ a multiplicative semiderivation of $R$ with associated a
nonzero epimorphism $g$ of $R.$

Also, we will make some extensive use of the basic commutator identities:

\begin{center}
$[x,yz]=y[x,z]+[x,y]z$

$[xy,z]=[x,z]y+x[y,z]$

$xo(yz)=(xoy)z-y[x,z]=y(xoz)+[x,y]z$

$(xy)oz=x(yoz)-[x,z]y=(xoz)y+x[y,z].$
\end{center}

\begin{lemma}
\label{lemma1}\cite[Lemma 2.1]{samman}Let $R$ be a semiprime ring, $I$ be a
nonzero ideal of $R$ and $a\in R$ such that $axa=0,$ for all $x\in I,$ then $%
a=0.$
\end{lemma}

\begin{theorem}
\label{theorem1}Let $R$ be a semiprime ring, $I$ be a nonzero ideal of $R$
and $d$ be a multiplicative semiderivation associated with a nonzero
epimorphism $g$ of $R$. If $d([x,y])=0,$ for all $x,y\in I,$ then $%
[x,d(x)]=0,$ for all $x\in I.$
\end{theorem}

\begin{proof}
By the hypothesis, we have 
\begin{equation}
d([x,y])=0,\text{ for\ all }x,y\in I.  \label{e1}
\end{equation}%
Replacing $yx$ by $y$ in (\ref{e1}) and using this, we get%
\begin{equation}
\lbrack x,y]d(x)=0,\text{ for all }x,y\in I.  \label{e2}
\end{equation}%
Writting $ry,r\in R$ for $y$ in (\ref{e2}) and using (\ref{e2}), we obtain
that 
\begin{equation}
\lbrack x,r]yd(x)=0,\text{ for all }x,y\in I,r\in R.  \label{e3}
\end{equation}%
Taking $yx$ by $y$ in (\ref{e3}), we have 
\begin{equation*}
\lbrack x,r]yxd(x)=0,\text{ for all }x,y\in I,r\in R.
\end{equation*}%
Right multipliying (\ref{e3}) with $x,$ we get%
\begin{equation*}
\lbrack x,r]yd(x)x=0,\text{ for all }x,y\in I,r\in R.
\end{equation*}%
Subtracting the last two equations, we arrive at%
\begin{equation*}
\lbrack x,r]y[x,d(x)]=0,\text{ for all }x,y\in I,r\in R.
\end{equation*}%
Replacing $d(x)$ by $r$ in this equation, we have%
\begin{equation*}
\lbrack x,d(x)]y[x,d(x)]=0,\text{ for all }x,y\in I.
\end{equation*}%
By Lemma 1, we get $[x,d(x)]=0,$ for all $x\in I.$ The proof is completed.
\end{proof}

\begin{theorem}
\label{theorem2}Let $R$ be a semiprime ring, $I$ be a nonzero ideal of $R$
and $d$ be a multiplicative semiderivation associated with a nonzero
epimorphism $g$ of $R$. If $d(xoy)=0,$ for all $x,y\in I,$ then $[x,d(x)]=0,$
for all $x\in I.$
\end{theorem}

\begin{proof}
By our hypothesis, we get%
\begin{equation}
d(xoy)=0,\text{ for all }x,y\in I.  \label{e4}
\end{equation}%
Writting $yx$ for $y$ in (\ref{e4}) and using (\ref{e4}), we obtain that 
\begin{equation}
(xoy)d(x)=0,\text{ for all }x,y\in I.  \label{5}
\end{equation}%
Substituting $ry,r\in R$ for $y$ in this equation and using this, we arrive
at 
\begin{equation*}
\lbrack x,r]yd(x)=0,\text{ for all }x,y\in I,r\in R.
\end{equation*}%
Using the same arguments after (\ref{e3}) in the proof of Theorem 1, we get
the required result.
\end{proof}

\begin{theorem}
\label{theorem3}Let $R$ be a semiprime ring, $I$ be a nonzero ideal of $R$
and $d$ be a multiplicative semiderivation associated with a nonzero
epimorphism $g$ of $R$. If $d(xy)\pm xy=0,$ for all $x,y\in I,$ then $%
[x,d(x)]=0,$ for all $x\in I.$
\end{theorem}

\begin{proof}
If $d=0,$ then we get $xy=0,$ for all $x,y\in I,$ and so $x\in I\cap
ann(I)=(0),$ for all $x\in I.$ Since $I$ is a nonzero ideal of $R$, we
assume that $d\neq 0.$

By our hypothesis, we get%
\begin{equation}
d(xy)\pm xy=0,\text{ for all }x,y\in I.  \label{e6}
\end{equation}%
Replacing $yz$ by $y$ in (\ref{e6}), we get%
\begin{equation}
(d(xy)\pm xy)z+g(xy)d(z)=0,
\end{equation}%
and so 
\begin{equation*}
g(xy)d(z)=0.
\end{equation*}%
That is%
\begin{equation}
g(x)g(y)d(z)=0,\text{ for all }x,y\in I.  \label{e7}
\end{equation}%
Taking $d(r)y,r\in R$ instead of $y$ in this equation and using\ $dg=gd$, it
reduces to%
\begin{equation*}
g(x)d(g(r))g(y)d(z)=0.
\end{equation*}%
Since $g$ is an epimorphism of $R,$ we have%
\begin{equation*}
g(x)d(r)g(y)d(z)=0,\text{ for all }x,y,z\in I.
\end{equation*}%
This implies that 
\begin{equation*}
g(x)d(z)g(y)d(z)=0,\text{ for all }x,y,z\in I.
\end{equation*}%
Writting $ty,t\in R$ for $y$ in this equation and using $g$ is surjective,
we obtain that%
\begin{equation*}
g(x)d(z)Rg(y)d(z)=(0),\text{ for all }x,y,z\in I.
\end{equation*}%
In particulary, we can write 
\begin{equation*}
g(x)d(z)Rg(x)d(z)=(0),\text{ for all }x,z\in I,r\in R
\end{equation*}%
and so%
\begin{equation*}
g(x)d(z)=0,\text{ for all }x,z\in I.
\end{equation*}%
Using this in the following equation, we have $d(xz)=d(x)z+g(x)d(z)=d(x)z,$
and so%
\begin{equation*}
d(xz)=d(x)z,\text{ for all }x,z\in I.
\end{equation*}%
Returning our hypothesis and using this, we find that%
\begin{equation}
(d(x)\pm x)y=0,\text{ for all }x,y\in I.  \label{e8}
\end{equation}%
and so 
\begin{equation*}
y(d(x)\pm x)Ry(d(x)\pm x)=(0),\text{ for all }x,y\in I.
\end{equation*}%
Since $R$ is semiprime ring, we get%
\begin{equation}
y(d(x)\pm x)=0,\text{ for all }x,y\in I.  \label{e9}
\end{equation}%
Comparing (\ref{e8}) and (\ref{e9}), we arrive at%
\begin{equation*}
\lbrack (d(x)\pm x),y]=0,
\end{equation*}%
and so%
\begin{equation*}
\lbrack (d(x)\pm x),x]=0.
\end{equation*}%
It reduces to%
\begin{equation*}
\lbrack d(x),x]=0,\text{ for all }x\in I.
\end{equation*}%
This completes the proof.
\end{proof}

\begin{theorem}
\label{theorem4}Let $R$ be a semiprime ring, $I$ be a nonzero ideal of $R$
and $d$ be a multiplicative semiderivation associated with a nonzero
epimorphism $g$ of $R$. If $d(xy)\pm yx=0,$ for all $x,y\in I,$ then $%
[x,d(x)]=0,$ for all $x\in I.$
\end{theorem}

\begin{proof}
If $d=0,$ then we get $yx=0,$ for all $x,y\in I$ and so $x\in I\cap
ann(I)=(0),$ for all $x\in I.$ Since $I$ is a nonzero ideal of $R$, we
assume that $d\neq 0.$

Assume that

\begin{equation}
d(xy)\pm yx=0,\text{\ for all }x,y\in I.  \label{e10}
\end{equation}%
Taking $yz$ instead of $y$ in this equation, we have%
\begin{equation*}
d(xy)z+g(xy)d(z)\pm yzx=0,\text{ for all }x,y,z\in I.
\end{equation*}%
For all $x,y,z\in I,$ we can write this equation%
\begin{equation*}
d(xy)z+g(xy)d(z)\pm yzx\mp yxz\pm yxz=0,\text{ for all }x,y,z\in I
\end{equation*}%
and so 
\begin{equation*}
(d(xy)\pm yx)z+g(xy)d(z)-y[x,z]=0,\text{ for all }x,y,z\in I.
\end{equation*}%
Using the hypothesis, we arrive at%
\begin{equation}
g(xy)d(z)-y[x,z]=0,\text{ for all }x,y,z\in I.  \label{e11}
\end{equation}%
Replacing $x$ by $z$ in (\ref{e11}) and using this, we get%
\begin{equation*}
g(xy)d(x)=0,\text{ for all }x,y\in I.
\end{equation*}%
Writting $d(t)ry,t,r\in R$ for $y$ in this equation and using $g$ is
surjective, we obtain that%
\begin{equation*}
g(x)g(d(t))Rg(y)d(x)=(0).
\end{equation*}%
Using $dg=gd,$ we have%
\begin{equation*}
g(x)d(g(t))Rg(y)d(x)=(0).
\end{equation*}%
We can write this equation using $g$ is surjective 
\begin{equation*}
g(x)d(r)Rg(y)d(x)=(0),
\end{equation*}%
and so%
\begin{equation*}
g(x)d(x)Rg(x)d(x)=(0).
\end{equation*}%
Hence we find that%
\begin{equation*}
g(x)d(x)=0,\text{ for all }x\in R.
\end{equation*}%
Now, let return (\ref{e11}). Writing $z$ by $y$ in this equation and using $%
g(z)d(z)=0,$ we arrive at%
\begin{equation}
z[x,z]=0,\text{ for all }x,z\in I.  \label{eq-11}
\end{equation}%
Replacing $yx$ by $x$ in this equation and using this, we get%
\begin{equation*}
zy[x,z]=0,\text{ for all }x,y,z\in I,
\end{equation*}%
and so%
\begin{equation}
zyr[x,z]=0,\text{ for all }x,z\in I,r\in R.  \label{eq-12}
\end{equation}%
(\ref{eq-11}) gives that%
\begin{equation}
yzr[x,z]=0,\text{ for all }x,y,z\in I,r\in R.  \label{eq-13}
\end{equation}%
Subtracting (\ref{eq-13}) form (\ref{eq-12}), we arrive at%
\begin{equation*}
\lbrack y,z]r[x,z]=0,\text{ for all }x,y,z\in I,r\in R,
\end{equation*}%
and so%
\begin{equation*}
\lbrack x,z]R[x,z]=(0),\text{ for all }x,,y,z\in I.
\end{equation*}%
By the semiprimeness of $R,$ we get%
\begin{equation*}
\lbrack x,z]=(0),\text{ for all }x,z\in I
\end{equation*}%
Using this equation in (\ref{e11}), we have%
\begin{equation*}
g(xy)d(z)=0,\text{ for all }x,y\in I.
\end{equation*}
This is the same as (\ref{e7}) in the proof of Theorem 3. Applying the same
arguments as we used in the last paragraph of the proof of Theorem 3, we get
the required result.
\end{proof}

\begin{theorem}
\label{theorem5}Let $R$ be a semiprime ring, $I$ be a nonzero ideal of $R$
and $d$ be a multiplicative semiderivation associated with a nonzero
epimorphism $g$ of $R$. If $d(x)d(y)\pm xy=0,$ for all $x,y\in I,$ then $%
[x,d(x)]=0,$ for all $x\in I.$
\end{theorem}

\begin{proof}
If $d=0,$ then we get $xy=0,$ for all $x,y\in I.$ We had done in the proof
of Theorem 3. So, we have $d\neq 0.$

By our hypothesis, we get%
\begin{equation}
d(x)d(y)\pm xy=0,\text{ for all }x,y\in I.
\end{equation}%
Replacing $yz$ by $y$ in this equation and using the hypothesis, we get%
\begin{equation*}
d(x)d(y)z+d(x)g(y)d(z)\pm xyz=0,
\end{equation*}%
and so%
\begin{equation*}
d(x)g(y)d(z)=0,\text{ for all }x,y,z\in I.
\end{equation*}%
Taking $ry,r\in R$ instead of $y$ in this equation and using $g$ is an
epimorphism, we have%
\begin{equation*}
d(x)Rg(y)d(z)=(0),
\end{equation*}%
and so%
\begin{equation*}
g(y)d(x)Rg(y)d(x)=(0),\text{ for all }x,y\in I.
\end{equation*}%
By the semiprimeness of $R$, we obtain that%
\begin{equation}
g(y)d(x)=0,\text{ for all }x,y\in I.  \label{e12}
\end{equation}%
Hence we get $d(xy)=d(x)y+g(x)d(y)=d(x)y,$and so%
\begin{equation}
d(xy)=d(x)y,\text{ for all }x,y\in I.  \label{e13}
\end{equation}%
On the other hand, right multiplying our hypothesis with $y,$ we get%
\begin{equation}
d(x)d(y)y\pm xy^{2}=0,\text{ for all }x,y\in I.  \label{e14}
\end{equation}%
Now, writing $xy$ by $x$ in the hypothesis and using\ (\ref{e13}), we find
that%
\begin{equation}
d(x)yd(y)\pm xy^{2}=0,\text{ for all }x,y\in I.  \label{e15}
\end{equation}%
Subtracting (\ref{e14}) from (\ref{e15}), we obtain that%
\begin{equation*}
d(x)[d(y),y]=0,\text{ for all }x,y\in I.
\end{equation*}%
Replacing $xz$ by $x$ in this equation and using this, we get%
\begin{equation*}
d(x)z[d(y),y]=0,\text{ for all }x,y,z\in I.
\end{equation*}%
It follows that%
\begin{equation*}
\lbrack d(y),y]z[d(y),y]=0,\text{ for all }x,y,z\in I.
\end{equation*}%
By Lemma 1, we get $[d(y),y]=0,$ for all $y\in I.$.
\end{proof}

\begin{theorem}
\label{theorem6}Let $R$ be a semiprime ring, $I$ be a nonzero ideal of $R$
and $d$ be a multiplicative semiderivation associated with a nonzero
epimorphism $g$ of $R$. If $d(x)d(y)\pm yx=0,$ for all $x,y\in I,$ then $%
[x,d(x)]=0,$ for all $x\in I.$
\end{theorem}

\begin{proof}
Using the same arguments begining of the proof of Theorem 3, we must have $%
d\neq 0.$

By our hypothesis, we get%
\begin{equation}
d(x)d(y)\pm yx=0,\text{ for all }x,y\in I.
\end{equation}%
Replacing $yx$ by $y$ in this equation and using this, we get%
\begin{equation*}
d(x)g(y)d(x)=0,\text{ for all }x,y\in I.
\end{equation*}%
Writing $ry,r\in R$ instead of $y$ in this equation and using $g$ is an
epimorphism, we have%
\begin{equation*}
d(x)Rg(y)d(z)=(0).
\end{equation*}%
In particulary, we get%
\begin{equation*}
g(y)d(x)Rg(y)d(x)=(0),
\end{equation*}%
and so%
\begin{equation*}
g(y)d(x)=0,\text{ for all }x,y\in I.
\end{equation*}%
Hence we have $d(xy)=d(x)y+g(x)d(y)=d(x)y,$and so $d(xy)=d(x)y,$ for all $%
x,y\in I.$

Now, right multiplying our hypothesis with $y,$ we get%
\begin{equation*}
d(x)d(y)y\pm yxy=0,\text{ for all }x,y\in I.
\end{equation*}%
Taking $xy$ by $x$ in the hypothesis and using $d(xy)=d(x)y,$ we have%
\begin{equation*}
d(x)yd(y)\pm yxy=0,\text{ for all }x,y\in I.
\end{equation*}%
Comparing the last two equations, we obtain that%
\begin{equation*}
d(x)[d(y),y]=0,\text{ for all }x,y\in I.
\end{equation*}

Applying the same arguments as used the end of the proof of Theorem 5, we
arrive at $[x,d(x)]=0,$ for all $x\in I.$
\end{proof}

\begin{theorem}
\label{theorem7}Let $R$ be a semiprime ring, $I$ be a nonzero ideal of $R$
and $d$ be a multiplicative semiderivation associated with a nonzero
epimorphism $g$ of $R$. If $d$ acts as a homomorphism on $I,$ then $%
[x,d(x)]=0,$ for all $x\in I.$
\end{theorem}

\begin{proof}
Assume that $f$ acts as a homomorphism on $R.$ Then one obtains%
\begin{equation}
d(xy)=d(x)y+g(x)d(y)=d(x)d(y),\text{ for all }x,y\in R.  \label{e16}
\end{equation}%
Replacing $y$ by $yz$ in (\ref{e16}) and using the hypothesis, we have%
\begin{equation*}
d(x)yz+g(x)d(yz)=d(xy)d(z).
\end{equation*}%
Since $d$ is multiplicative semiderivation of $R,$ we get%
\begin{equation*}
d(x)yz+g(x)d(yz)=d(x)yd(z)+g(x)d(y)d(z)
\end{equation*}%
and so%
\begin{equation*}
d(x)yz+g(x)d(yz)=d(x)yd(z)+g(x)d(yz).
\end{equation*}%
That is%
\begin{equation}
d(x)y(z-d(z))=0,\text{ for all }x,y,z\in I.  \label{eq-16}
\end{equation}%
It follows that%
\begin{equation}
d(x)yd(w)(z-d(z))=0,\text{ for all }x,y,z,w\in I.  \label{e17}
\end{equation}%
Returning (\ref{e16}), we can write%
\begin{equation*}
d(x)y+g(x)d(y)=d(x)d(y).
\end{equation*}%
That is%
\begin{equation}
d(x)(y-d(y))=-g(x)d(y),\text{ for all }x,y,z\in I.  \label{e18}
\end{equation}%
We can write from (\ref{e17}) using (\ref{e18}) 
\begin{equation*}
d(x)yg(w)d(z)=0
\end{equation*}%
and so%
\begin{equation*}
g(w)d(x)yRg(w)d(x)y=(0),\text{ for all }x,y,w\in I.
\end{equation*}%
Since $R$ is semiprime, we conclude that $g(w)d(x)I=0,$ for all $x,w\in I.$

Now, right multiplying our hypothesis with $y$ and using $g(x)d(y)y=0,$ we
get%
\begin{equation*}
d(x)(y-d(y))y=0,\text{ for all }x,y\in I.
\end{equation*}%
By (\ref{eq-16}), we can write 
\begin{equation*}
d(x)y(y-d(y))=0,\text{ for all }x,y,z\in I.
\end{equation*}%
Subtracting the last two equations, we get%
\begin{equation}
d(x)[d(y),y]=0,\text{ for all }x,y\in I.  \label{eq-19}
\end{equation}%
Replacing $x$ with $xz$ in (\ref{eq-19}) and using this, we obtain 
\begin{equation}
d(x)z[d(y),y]=0,\text{ for all }x,y\in I  \label{eq-20}
\end{equation}%
which yields that%
\begin{equation*}
xd(x)z[d(y),y]=0,\text{ for all }x,y\in I.
\end{equation*}%
Taking $z$ by $xz$ in (\ref{eq-20}), we get%
\begin{equation*}
d(x)xz[d(y),y]=0,\text{ for all }x,y\in I.
\end{equation*}%
Subtracting the last two equations, we find that 
\begin{equation*}
\lbrack d(x),x]z[d(y),y]=(0),\text{ for all }x,y,z\in I.
\end{equation*}%
In particular, 
\begin{equation*}
\lbrack d(x),x]z[d(x),x]=(0),\text{ for all }x,z\in I.
\end{equation*}%
By Lemma 1, we have $[d(x),x]=(0),$ for all $x\in I.$
\end{proof}

\begin{theorem}
\label{theorem8}Let $R$ be a semiprime ring, $I$ be a nonzero ideal of $R$
and $d$ be a multiplicative semiderivation associated with a nonzero
epimorphism $g$ of $R$. If $d$ acts as an anti-homomorphism on $I,$ then $%
[x,d(x)]=0,$ for all $x\in I.$
\end{theorem}

\begin{proof}
We have%
\begin{equation}
d(xy)=d(x)y+g(x)d(y)=d(y)d(x),\text{ for all }x,y\in I.  \label{e21}
\end{equation}%
Taking $y$ by $xy$ in this equation, we get%
\begin{equation*}
d(x)xy+g(x)d(xy)=d(xy)d(x).
\end{equation*}%
Since $d$ is a multiplicative semiderivation of $R,$ we have%
\begin{equation*}
d(x)xy+g(x)d(xy)=d(x)yd(x)+g(x)d(y)d(x).
\end{equation*}%
Using the hypothesis, we arrive at%
\begin{equation}
d(x)xy=d(x)yd(x),\text{ for all }x,y\in I.  \label{e22}
\end{equation}%
Replacing $y$ with $yx$ in (\ref{e22}) and using this, we obtain%
\begin{equation}
d(x)y[d(x),x]=0,\text{ for all }x,y\in I.  \label{e23}
\end{equation}%
Left multiplying this equation by $x,$ we get%
\begin{equation*}
xd(x)y[d(x),x]=0,\text{ for all }x,y\in I.
\end{equation*}%
Writing $y$ by $xy$ in (\ref{e23}), we have%
\begin{equation*}
d(x)xy[d(x),x]=0,\text{ for all }x,y\in I.
\end{equation*}%
Subtracting the last two equations, we find that%
\begin{equation*}
\lbrack d(x),x]y[d(x),x]=0,\text{ for all }x,y\in I.
\end{equation*}%
By Lemma 1, we have $[d(x),x]=(0),$ for all $x\in I.$ This completes the
proof.
\end{proof}

\end{document}